\documentclass{article}
\usepackage[utf8]{inputenc}

\usepackage{ifxetex, ifluatex, amsmath, mathtools, amsfonts}
\if\ifxetex T\else\ifluatex T\else F\fi\fi T%
  \usepackage{fontspec}
\else
  \usepackage[T1]{fontenc}
  \usepackage[utf8]{inputenc}
  \usepackage{lmodern}
\fi

\def\R{\mathbb{R}}
\def\mc{\mathcal}
\def\mb{\mathbb}
\def\Lt{L^2([0,1])}

\newtheorem{lemma}{Lemma}
\newtheorem{theorem}{Theorem}

\newtheorem{corollary}{Corollary}

\newenvironment{proof}{\paragraph{Proof:}}{\hfill$\square$}

\usepackage{hyperref}

\begin{document}

\title{Numerical Approximation of Andrews Plots with Optimal Spatial-Spectral Smoothing}
\author{Mitchell Rimerman$^1$, Nate Strawn$^1$}

\date{%
    $^1$Georgetown University\\
    \today
}

\maketitle

\begin{abstract}
Andrews plots provide aesthetically pleasant visualizations of high-dimensional datasets. This work proves that Andrews plots (when defined in terms of the principal component scores of a dataset) are optimally ``smooth'' on average, and solve an infinite-dimensional quadratic minimization program over the set of linear isometries from the Euclidean data space to $L^2([0,1])$. By building technical machinery that characterizes the solutions to general infinite-dimensional quadratic minimization programs over linear isometries, we further show that the solution set is (in the generic case) a manifold. To avoid the ambiguities presented by this manifold of solutions, we add ``spectral smoothing'' terms to the infinite-dimensional optimization program to induce Andrews plots with optimal spatial-spectral smoothing. We characterize the (generic) set of solutions to this program and prove that the resulting plots admit efficient numerical approximations. These  spatial-spectral smooth Andrews plots tend to avoid some ``visual clutter'' that arises due to the oscillation of trigonometric polynomials. 
\end{abstract}

\section{Introduction}

Data visualization techniques form an essential toolkit for scientists seeking to glean insights from data, and scatterplots constitute an incredibly straightforward tool for building intuition from data. However, scatter plots ``compress" data in a lossy manner, which means such intuitions need careful vetting. Andrews plots \cite{andrews1972plots} offer a way to double-check insights gleaned from conventional 2D or 3D scatterplots because they completely preserve the information within a dataset. Figure \ref{fig1} displays an example. While 2D and 3D scatterplots yield visualizations where data points are represented by non-overlapping objects (points or dots), embedding high-dimensional data in 2D or 3D necessarily collapses distances between points in general. Andrews plots trade off ``visual clutter'' for fidelity of distances. 

\begin{center}
\begin{figure}
\centering
\includegraphics[scale=0.4]{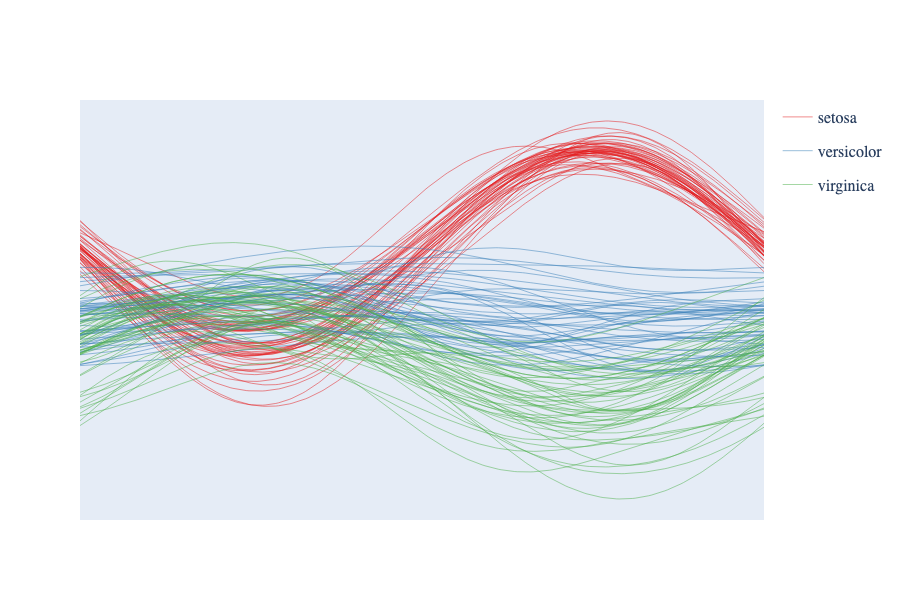}
\caption{An example of Andrews plots for Anderson's Iris dataset \cite{anderson1936species} .}
\label{fig1}
\end{figure}
\end{center}

In \cite{andrews1972plots}, Andrews suggests a map from $\Phi:\R^d\to L^2([0,1])$ of the form $x\mapsto f_x(t)$ where
\begin{align*}
f_x(t)=&(u_1^T x) + (u_2^T x) \sqrt{2}\cos(2\pi t) + (u_3^T x) \sqrt{2}\sin(2\pi t)\\
& + (u_4^T x)  \sqrt{2}\cos(4\pi t) + (u_5^T x) \sqrt{2}\sin(4\pi t) +\cdots
\end{align*}
where $\{u_i\}_{i=1}^d$ is an orthonormal basis of $\R^n$. Such maps are isometries from $\R^d$ to $L^2([0,1])$. In \cite{andrews1972plots}, Andrews suggest using the orthonormal basis from principal component analysis (PCA) to maintain compatibility with the hypothesis tests and confidence intervals he develops. 

The work \cite{strawn2022filament} develops 3D extensions of Andrews plots that mitigate ``visual clutter'' issues inherent to Andrews plots by plotting 1D curves in a three-dimensional ``drag-to-rotate'' environment. The theory developed there proves that an orthonormal basis of principal component vectors  arises when we attempt to minimize the ``mean quadratic variation" over all the plots. That is, a principal component analysis basis produces ``optimally smooth'' 2D Andrews plots, and this optimality holds over the entire set of isometries from $\R^d$ to $L^2([0,1])$. 

The work \cite{strawn2022filament} also demonstrated that the ``optimally smooth'' 2D Andrews plots generically form a manifold parametrized by a product of circles (a generalized torus). This provides the opportunity to obtain optimally smooth 2D Andrews plots that also admit an asymptotic ``tour" property ensuring that the map $x\to f_x(t_0)$ is nearly a projection from $\R^d$ to $\R^2$ for all $t_0$. The tour property ensures that plots avoid ``flocking'' behavior that causes plots to ``braid'' together. 

In \cite{strawn2022filament}, it was suggested that a substantially similar proof suffices to show that standard Andrews plots with PCA coefficients also minimize the mean quadratic variation functional. However, much like the case of 2D Andrews plots, there are degrees of freedom within the set of such minimizers. 

\subsection{Our Contributions}

In this paper, we provide a characterization of the minimizers of the mean (across a dataset) quadratic variation of isometries from $\R^d$ to $\mc{H}$, where $\mc{H}$ is the Sobolev space
\[
\mc{H} = \{f\in \Lt: f^\prime\in \Lt\}.
\] 
This characterization includes the plots involving PCA coefficients originally suggested by Andrews \cite{andrews1972plots}. We prove this characterization via a new theoretical technique that offers the possibility of further extensions. Our main technique, Theorem \ref{thm1}, characterizes the solutions to certain infinite-dimensional quadratic optimization programs on isometries from $\R^d$ to $\ell^2(\mb{Z})$. By utilizing the inverse discrete-time Fourier transform, we specialize this result to characterize the minimizers of the mean quadratic variation in Corollary \ref{corMMQV}. 

 This characterization reveals numerous degrees of freedom due to the translation-invariant subspaces
\[
\text{span}\{\cos(2\pi k t), \sin(2\pi k t)\}.
\]
In particular, any map of the form
\begin{align*}
x\to (u_1^T x)&+ \sqrt{2}(u_2^Tx)(a_1\cos(2\pi t)+b_1\sin(2\pi t))\\
& + \sqrt{2}(u_2^Tx)(b_1\cos(2\pi t)-b_1\sin(2\pi t))\\
& + \sqrt{2}(u_3^Tx)(a_2\cos(4\pi t)+b_2\sin(4\pi t))\\
& + \sqrt{2}(u_4^Tx)(b_2\cos(4\pi t)-a_2\sin(4\pi t))\\
& + \cdots
\end{align*}
(where $a_i^2+b_i^2=1$ for all $i$) constitutes a linear isometry which minimizes the mean quadratic variation of the dataset $\{x_i\}_{i=1}^n\subset\R^d$ over the set of linear isometries from $\R^d$ to $\mc{H}$. 

These degrees of freedom motivate the second component of our contributions. Heuristically, the oscillation of trigonometric polynomials contributes to multiple crossings of Andrews plots. Smoothing Fourier coefficients dampens these oscillations. Therefore, for $\alpha>0$, we define the spatial-spectral quadratic variation of $f\in L^2([0,1])$ by
\[
\frac{\alpha}{4\pi^2}\Vert f^\prime\Vert_{\Lt}^2 + \sum_{k\in\mb{Z}} \vert \hat f[k+1]-\hat f[k]\vert^2
\]
where the $\hat f[k]$ are Fourier coefficients of $f$.  We show that minimizing the mean spatial-spectral quadratic variation over the set of isometries from $\R^d$ to $\mc{H}$ fits in the framework of Theorem \ref{thm1}, and also that the solutions to this program may be approximated in an efficient manner. These results are encapsulated in Theorem \ref{thm2}. 

\begin{center}
\begin{figure}
\centering
\includegraphics[scale=0.4]{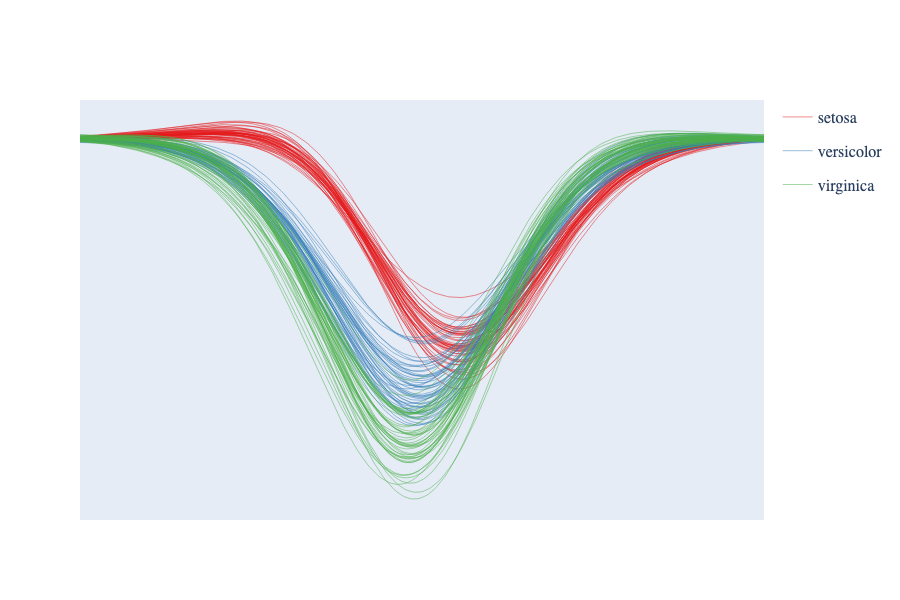}
\caption{Andrews plots with optimal spatial-spectral smoothing for  Anderson's Iris dataset.}
\label{fig2}
\end{figure}
\end{center}

\subsection{Organization}
In Section 2, we prove our main technical tool, Theorem \ref{thm1} and use it to characterize the minimizers of the mean quadratic variation in Corollary \ref{corMMQV}. In Section 3, we discuss the minimum mean spatial-spectral quadratic variation problem, prove that it fits in the framework of Theorem \ref{thm1}, and then prove that the optimal isometries admit numerical approximations. These results are summarized in Theorem \ref{thm2}. In Section 4, we provide several examples for different datasets. Section 5 concludes the paper with some interesting open problems.

\subsection{Notation}

Let $\{x_i\}_{i=1}^n\subset\R^d$ denote a dataset, and let 
\[
X= \begin{pmatrix}
x_1 & x_2 & \cdots & x_n
\end{pmatrix}
\]
denote the data matrix with data in columns. We let $\ell^2=\ell^2(\mb{Z})$ denote the Hilbert space of (possibly complex) square summable sequences over the set of integers $\mb{Z}$. We let $\mc{I}$ denote the collection of linear isometries from $\R^d\to \ell^2$. For $x,y\in\R^d$, let $x^Ty$ denote the standard inner product on $\R^d$, and for $f,g\in \ell^2$, we let 
\[
\langle f,\:g\rangle_{\ell^2} = \sum_{k\in\mb{Z}}f[k]g[k]
\]
We consider general optimization programs of the form
\begin{align}
\min_{\Phi\in \mc{I}} \sum_{i=1}^n \langle A \Phi(x_i), \Phi(x_i)\rangle_{\ell^2}\label{genOpt}
\end{align}
where $A$ is any self-adjoint, positive definite operator on $\ell^2$ with a purely discrete spectrum (see \cite{schmudgen2012unbounded}). We specify the domain of $A$ by
\[
\text{dom}(A) = \{f\in \ell^2: Af\in\ell^2\}
\]
and let $\sigma(A)$ denote the spectrum of $A$ (that is, $\lambda\in\sigma(A)$ if $A-\lambda I$ is not surjective or not injective). Since we only consider cases where $A$ has purely discrete spectrum, each $\lambda\in\sigma(A)$ is an eigenvalue of $A$ with finite multiplicity. We let 
\[
\{\lambda_i(A)\}_{i=1}^d \subset \sigma(A)
\]
denote the ``$d$ lowest eigenvalues of $A$'' (counting multiplicity) where
\[
\lambda_1(A)\leq \lambda_2(A)\leq \cdots \leq \lambda_d(A).
\]
We also let 
\[
\sigma_1(X) \geq \sigma_2(X) \geq \cdots \geq \sigma_d(X)
\]
denote the singular values of a data matrix $X$.

\section{Optimal Isometries from $\R^d$ to $\ell^2(\mb{Z})$}

We prove Theorem \ref{thm1} in this section, which characterizes solutions to certain quadratic minimization programs over spaces of isometries from $\R^d$ to $\ell^2$. First, Lemma \ref{lemIsoFac} helps us represent such isometries in terms of an orthonormal basis of left-singular vectors and an orthonormal collection in $\ell^2$.

\begin{lemma}\label{lemIsoFac}
If $\Phi$ is a linear isometry from $\R^d$ to some Hilbert space $h$ and $\{u_i\}_{i=1}^d$ is some orthonormal basis of $\R^d$, then there is an orthonormal collection $\{\phi_i\}_{i=1}^d\subset H$ such that
\[
\Phi(x) = \sum_{i=1}^d (u_i^T x) \phi_i
\]
for all $x\in\R^d$.
\end{lemma}

\begin{proof}
Set $\phi_i=\Phi(u_i)$ for $i=1,\ldots, d$ and observe that $\{\phi_i\}_{i=1}^d$ inherits orthonormality from $\{u_i\}_{i=1}^d$ via the isometry property of $\Phi$. Linearity gives us
\[
\Phi(x) = \Phi\left( \sum_{i=1}^d (u_i^T x) u_i\right)=  \sum_{i=1}^d (u_i^T x) \Phi(u_i) =  \sum_{i=1}^d (u_i^T x) \phi_i
\]
for all $x\in\R^d$.
\end{proof}\\

Lemma \ref{lemIsoFac} allows us to represent infinite-dimensional quadratic programs over spaces of projections as infinite-dimensional quadratic programs subject to a finite number of quadratic (orthogonality) conditions. We now state an prove Theorem \ref{thm1}.

\begin{theorem}\label{thm1}
Let $\{x_i\}_{i=1}^n\subset\R^d$ and suppose $A$ is a self-adjoint, positive semi-definite on $\ell^2$ with a purely discrete spectrum. Then
\[
\sum_{i=1}^n \langle A \Phi(x_i), \Phi(x_i)\rangle_{\ell^2} \geq \sum_{i=1}^d \lambda_i(A)\sigma_i(X)^2,
\]
for all $\Phi\in\mc{I}$, and this lower bound is attained for all $\Phi$ such that
\begin{align}
\Phi(x) = \sum_{i=1}^d (u_i^T x)\phi_i\label{eqnIsoFac}
\end{align}
where $\{u_i\}_{i=1}^d$ form an orthonormal basis of left singular vectors of $X$ corresponding to singular values $\{\sigma_i\}_{i=1}^d$ in a respective fashion, and $\{\phi_i\}_{i=1}^d$ form an orthonormal set of eigenvectors of $A$ corresponding to the eigenvalues $\{\lambda_i\}_{i=1}^d$ in a respective fashion. If the singular values $\{\sigma_i(X)\}_{i=1}^d$ are distinct and non-zero, any minimizer $\Phi$ satisfies (\ref{eqnIsoFac}) with $\{\phi_i\}_{i=1}^d$ an orthonormal set of eigenvectors of $A$ corresponding to the eigenvalues $\{\lambda_i\}_{i=1}^d$.
\end{theorem}

\begin{proof}
For any linear isometry $\Phi\in\mc{I}$, set $\phi_i = \Phi(u_i)$ where $\{u_i\}_{i=1}^d\subset\R^d$ is the orthonormal basis of left singular vectors stated in the theorem. By Lemma \ref{lemIsoFac}, there is an orthonormal collection $\{\phi_i\}_{i=1}^d\subset \ell^2 $ such that
\[
\Phi(x) = \sum_{i=1}^d (u_i^T x)\phi_i
\]
for all $x\in\R^d$. Then
\[
\langle A \Phi(x),\: \Phi(x)\rangle = \sum_{i=1}^d \sum_{j=1}^d (u_i^Tx)(u_j^Tx) \langle A \phi_i,\: \phi_j\rangle = \text{trace}(U^Txx^T U [\langle A \phi_i,\:\phi_j\rangle]),
\]
and summing over the dataset yields
\[
\sum_{i=1}^n \langle A \Phi(x_i),\: \Phi(x_i)\rangle =  \text{trace}(U^TXX^T U [\langle A \phi_i,\:\phi_j\rangle])=\sum_{i=1}^d \langle A\phi_i,\phi_i\rangle\sigma_i(X)^2.
\]
Set $D_i = \langle A\phi_i,\phi_i\rangle$ and $\sigma_i=\sigma_i(X)$ and apply summation by parts to obtain
\[
\sum_{i=1}^d \langle A\phi_i,\phi_i\rangle\sigma_i(X)^2 = \sigma_d^2\sum_{i=1}^dD_i + \sum_{k=1}^{d-1} (\sigma_i^2-\sigma_{i+1}^2)\sum_{j=1}^k D_j
\]
On the other hand, \cite{bownik2017schur} (or more specifically \cite{schur1923uber}), $\sum_{j=1}^k D_j \geq \sum_{j=1}^k \lambda_j(A)$, and the non-increasing nature of $\sigma_i$ yields
\[
\sum_{i=1}^d \langle A\psi_i,\psi_i\rangle\sigma_i(X)^2 \geq \sigma_d^2\sum_{i=1}^d\lambda_i(A) + \sum_{k=1}^{d-1} (\sigma_k^2-\sigma_{k+1}^2)\sum_{j=1}^k \lambda_i(A) = \sum_{i=1}^d \lambda_i(A)\sigma_i(X)^2.
\]
Note that projections of the form (\ref{eqnIsoFac}) satisfy this lower bound.

On the other hand, the lower bound is attained if and only if $\sum_{i=1}^d D_i=\sum_{i=1}^d\lambda_i(A)$ if $\sigma_d(X)\not=0$ and
\[
\sum_{i=1}^k D_i = \sum_{j=1}^k \lambda_i(A)
\]
for all $k=1,\ldots, d-1$ such that $\sigma_k^2\not=\sigma_{k+1}^2$. By the spectral theorem for self-adjoint unbounded operators from \cite{schmudgen2012unbounded} (and since we have assumed that $A$ has a purely discrete spectrum), if the $\sigma_k$'s are all unique and non-zero, we have that the $\phi_i$ in $(\ref{eqnIsoFac})$ must be eigenvectors of $A$ with eigenvalues $\lambda_i(A)$ via an induction argument.
\end{proof}

\paragraph{Discussion} Without a discrete spectrum, one can show that the lower bound still exists, but it may never be attainable. Additionally, if the singular values of $X$ are not distinct, it is still possible to characterize the minimizers of (\ref{genOpt}), but the level of complexity increases substantially.\\

\subsection{Optimal mean quadratic variation of Andrews plots using PCA scores}

We now let $\mc{I}(\R^d,\:\mc{H})$ denote the linear isometries from $\R^d$ to $\Lt$ whose images are in $\mc{H}$. While $\mc{I}(\R^d,\:\mc{H})$ is not closed in the operator norm topology inherited from $\Lt$, our approach to optimization does not require any consideration of limits. We let $\mc{F}^{-1}: \Lt\to \ell^2$ denote the inverse discrete time Fourier transform, and use $\hat f$ to denote $\mc{F}^{-1} f$ for any $f\in \Lt$. Recall the following properties of $\mc{F}^{-1}$:
\begin{enumerate}
\item With $\hat f[k]$ the $k$th coordinate of $\hat f$, we have
\[
f(t) = \sum_{k\in\mb{Z}} \hat f[k]e^{2\pi \iota k t}
\]
where $\iota=\sqrt{-1}$.
\item If $f, f^\prime\in L$, then $\widehat{f^\prime}\in \ell^2$ and
\[
\widehat{f^\prime}[k] = 2\pi \iota k \hat f[k]
\]
for all $k\in\mb{Z}$. 
\item When $f$ is real valued,
\[
\hat f[-k] = \overline{\hat f[k]}.
\]
\end{enumerate}

We let $\text{span}_{\mb{R}}$ and $\text{span}_{\mb{C}}$ denote the span of vectors using real coefficients and complex coefficients respectively. We now state and prove the characterization of minimizers of the mean quadratic variation. In particular, observe that the set of solutions always includes standard Andrews plots where the coefficients come from PCA scores.

\begin{corollary}\label{corMMQV}
Suppose $X\in\R^{d\times n}$ is a data matrix and $\{u_i\}_{i=1}^d$ is an orthonormal basis of left singular vectors of $X$ associated with the singular values $\{\sigma_i(X)\}_{i=1}^d$ in a respective fashion. Let $D$ denote the differential operator ($f^\prime = D f$) on $\mc{H}$. For the minimum mean quadratic variation problem
\[
\min_{\Psi\in \mathcal{I}(\R^d, \mc{H})} \frac{1}{n} \sum_{i=1}^n\Vert D \Phi(x_i)\Vert^2,
\]
the lower bound is $\frac{1}{n}\sum_{k=2}^d 4\pi^2 \lfloor k/2\rfloor^2\sigma_k(X)^2$. This lower bound is attained for all $\Phi$ of the form
\[
\Phi(x) = \pm(u_1^Tx){\bf 1} + \sum_{k=2}^d (u_k^T x) \phi_k
\]
where $\{\phi_{2j}, \phi_{2j+1}\}$ is an orthonormal basis for
\[
\text{span}\{\sqrt{2}\cos(2\pi j t),\: \sqrt{2}\sin(2\pi j t)\}
\]
for all $j\geq 1$. These are the only solutions if the $\sigma_i(X)$'s are all unique and non-zero.
\end{corollary}

\begin{proof}
The Plancheral theorem gives us
\begin{align}
\Vert D f\Vert_{\Lt}^2 = \Vert \mc{F}^{-1} Df\Vert_{\ell^2}^2= \sum_{k\in\mb{Z}} \vert 2\pi i k \hat f[k]\vert^2 = \sum (4\pi^2 k^2 \hat f[k])\overline{\hat f[k]} = \langle A \hat f, \hat f\rangle\label{mmqvEq}
\end{align}
for the diagonal operator $A$ with $A_{k,k} = 4\pi^2 k^2$ for all $k\in\mb{Z}$. By Lemma \ref{lemIsoFac}, $\Phi\in \mc{I}(\R^d, \mc{H}) \subset \mc{I}(\R^d, \Lt)$ satisfies
\[
\Phi(x) = \sum_{i=1}^d (u_i^Tx) \phi_i
\]
for all $x\in\R^d$ where $\{u_i\}_{i=1}^d$ is an orthonormal basis of left singular vectors of $X$ and $\{\phi_i\}_{i=1}^d\subset \mc{H}$ is orthonormal in the inner product on $\Lt$. Next note that $\widehat{\Phi(x)} = \sum_{i=1}^d (u_i^Tx)\hat\phi_i$ and $\Psi\in \mc{I}$ if $\Psi(x) = \sum_{i=1}^d (u_i^Tx)\hat\phi_i$. Given this identification and the equality (\ref{mmqvEq}), we conclude that $\Phi$ minimizes
\[
\sum_{i=1}^n \Vert D \Phi(x_i)\Vert_{\Lt}^2
\]
over $\Phi\in \mc{I}(\R^d, \mc{H})$ if and only if $\Psi$ minimizes
\[
\sum_{i=1}^n \langle A \Psi(x_i),\: \Psi(x_i)\rangle
\]
over $\Psi\in \mc{L}$ subject to the additional condition that 
\[
\widehat{\Psi(x)}[-k] = \overline{\widehat{\Psi(x)}[k]}
\]
for all $k\in\mb{Z}$. 

Since $A$ is diagonal, it is trivially self-adjoint and $\sigma(A)=\{0, 4\pi^2, 16\pi^2,\ldots\}$ where $0$ has multiplicity $1$ and the other eigenvalues have multiplicity $2$. These observations show that $A$ is self-adjoint, positive semi-definite, and has a purely discrete spectrum, and hence Theorem \ref{thm1} applies to the optimization program
\[
\min_{\Psi\in \mc{I}} \sum_{i=1}^n \langle A \Psi(x_i), \Psi(x_i)\rangle.
\]
However, note that this holds over all linear isometries from $\R^d$ to complex sequences in $\ell^2$. Since this set of linear isometries is larger than the set of linear isometries having the form of $\Psi$ above, we conclude
\[
\sum_{i=1}^n \Vert D \Phi(x_i)\Vert_{\Lt}^2\geq\sum_{k=2}^d 4\pi^2 \lfloor k/2\rfloor^2\sigma_k(X)^2
\]
from the lower bound in Theorem \ref{thm1}.

Let $\{e_i\}_{i\in\mathbb{Z}}$ denote the canonical orthonormal basis of $\ell^2$, and note that $\{\psi_i\}_{i=1}^d$ is an orthonormal basis of eigenvectors of $A$ with eigenvalues 
\[
0, 4\pi^2, 4\pi^2, 16\pi^2, 16\pi^2, \ldots
\]
if and only if $\psi_1\in\text{span}_{\mb{C}}(e_0)$ and $\psi_{2j}, \psi_{2j+1}\in\text{span}_{\mb{C}}\{e_{-j}, e_j\}$. for $j\geq 1$. If we further impose that $\psi_i = \hat\phi_i$ for some real $\phi_i\in\Lt$, we get that $\psi_1= \pm e_0$ and $\psi_{2j}, \psi_{2j+1}\in\text{span}_{\mb{R}}\{(e_j+e_{-j})/\sqrt{2}, (e_j-e_{-j})/\sqrt{2}\iota\}$. for $j\geq 1$. For $\phi_i\in\Lt$ with $\hat \phi_i = \psi_i$, it follows that $\phi_0 = \pm {\bf 1}$ and $\{\phi_{2j}, \phi_{2j+1}\}$ is an orthonormal basis for
\[
\text{span}\{\sqrt{2}\cos(2\pi j t),\: \sqrt{2}\sin(2\pi j t)\}.
\]
This shows that Theorem \ref{thm1} implies the desired form for some minimizers of the mean quadratic variation. Finally, if the singular values of $X$ are unique and non-zero, then these are the only minimizers by Theorem \ref{thm1}.
\end{proof}\\

\section{Optimal Spatial-Spectral Andrews plots}

We introduce the spatial-spectral quadratic variation and prove Theorem \ref{thm2} in this section. For a given parameter $\alpha>0$ and a dataset $\{x_i\}_{i=1}^n\subset\R^d$, we define the {\bf minimum mean spatial-spectral quadratic variation program} over $\Phi\in \mc{I}(\R^d, \mc{H})$ by
\begin{align}
\min_{\Phi\in \mc{I}(\R^d, \mc{H})} \frac{1}{n} \sum_{i=1}^n\left( \frac{\alpha}{4\pi^2} \Vert D \Phi(x_i)\Vert_{\Lt}^2 + \sum_{k\in\mb{Z}}\left\vert\widehat{\Phi(x_i)}[k+1]-\widehat{\Phi(x_i)}[k]\right\vert^2\right)
\end{align}
These second terms under the sum measure ``discrete'' quadratic variation in the Fourier domain. These terms can be represented in the spatial domain:
\begin{align*}
\sum_{k\in\mb{Z}}\left\vert\hat f[k+1]-\hat f[k]\right\vert^2&= \sum_{k\in\mb{Z}}(-\hat f[k-1]+2\hat f[k]- \hat f[k+1])\overline{\hat f[k]}\\
& = 2\int_0^1(1-\cos(2\pi t))\vert f(t)\vert^2\:dt.
\end{align*}
This form indicates that this term promotes functions that concentrate near $t=0$. We first show that this program fits in the framework provided by Theorem \ref{thm1} by considering the representation of the spatial-spectral quadratic variation in the Fourier coefficient domain. Lemma \ref{lemEvenOdd} establishes this representation, and ensures that the resulting operators are self-adjoint and discrete. 

To state and prove Lemma \ref{lemEvenOdd}, we need to establish some more notation. Letting $\{e_k\}_{k\in\mb{Z}}$ again denote the standard orthonormal basis on $\ell^2$, we let
\[
\mc{E} = \{e_0\}\cup\{(e_k+e_{-k})/\sqrt{2}\}_{k\geq 1}\text{ and } \mc{O}= \{(e_k-e_{-k})/\sqrt{2}\iota\}_{k\geq 1}
\]
and observe that $v\in\ell^2$ satisfying $v[-k]=\overline{v[k]}$ maps isometrically to the real vector space $\ell^2(\mb{N}_0)\oplus \ell^2(\mb{N})\cong \ell^2(\mb{Z})$ (where $\mb{N}_0=\{0\}\cup\mb{N}$) via $v\to (v_{\mc{E}}, v_{\mc{O}})$ where
\[
v_{\mc{E}} (\langle v,\: e_0\rangle, \langle v,\: (e_1+e_{-1})/\sqrt{2}\rangle,\ldots
\]
and
\[
 v_{\mc{O}}) = \langle v,\:(e_1-e_{-1})/\sqrt{2}\iota\rangle,\:\langle v,\: (e_2-e_{-2})/\sqrt{2}\iota\rangle ,\ldots.
\]
We call this isometry the {\bf even-odd isometry}. 

\begin{lemma}\label{lemEvenOdd}
For any $\alpha>0$ and $f\in\mc{H}$, we have
\begin{align*}
\frac{\alpha}{4\pi^2}\Vert D f \Vert_{\Lt}^2 + \sum_{k\in\mb{Z}}\left\vert\widehat{f}[k+1]-\widehat{f}[k]\right\vert^2& = \langle B_\alpha \hat f,\:\hat f\rangle\\
&= \langle B_\alpha^\mc{E}\hat f_\mc{E},\:\hat f_\mc{E}\rangle_{\ell^2(\mb{N}_0)}+\langle B_\alpha^\mc{O}\hat f_\mc{O},\:\hat f_\mc{O}\rangle_{\ell^2(\mb{N})}
\end{align*}
where $B_\alpha$ is an infinite tridiagonal array with diagonal entries $(B_\alpha)_{k,k}=\alpha k^2+2$ and $(B_\alpha)_{k+1,k} = (B_\alpha)_{k-1,k}=-1$ for all $k\in\mb{Z}$, and where
\[
B_\alpha^\mathcal{E} = \begin{pmatrix}
2 & -\sqrt{2} & 0 & 0 & \ldots \\
-\sqrt{2} & \alpha + 2 & -1 & 0 & \ldots\\
0 & -1 & 4\alpha + 2 & -1 & \ldots\\
0 & 0 & -1 & 9\alpha + 2 & \ldots\\
\vdots & \vdots & \vdots & \vdots & \ddots 
\end{pmatrix}
\]
and
\[
B_\alpha^\mathcal{O} = \begin{pmatrix}
 \alpha + 2 & -1 & 0 & \ldots\\
 -1 & 4\alpha + 2 & -1 & \ldots\\
 0 & -1 & 9\alpha + 2 & \ldots\\
 \vdots & \vdots & \vdots & \ddots 
\end{pmatrix}
\]
are Jacobi operators. Moreover, the operators $B_\alpha^{\mc{E}}$ and $B_\alpha^{\mc{O}}$ are self-adjoint and positive semi-definite, and $\sigma(B_\alpha^{\mc{E}})$ and $\sigma(B_\alpha^{\mc{O}})$ are both discrete.
\end{lemma}

\begin{proof}
The forms of $B_\alpha$, $B_\alpha^\mc{E}$, and $B_\alpha^\mc{O}$ follow by a straightforward computation. Now consider the note (following Theorem 4.5) in \cite{petropoulou2014self} about Theorem 4.1 in \cite{janas2001multithreshold}: for a Jacobi operator
\[
\begin{pmatrix}
b_1 & a_1 & 0 & 0 & \cdots\\
a_1 & b_2 & a_2 & 0 & \cdots\\
0 & a_2 & b_3 & a_3& \cdots\\
0 & 0 & a_3 & b_4& \cdots\\
\vdots & \vdots & \vdots &\vdots &\ddots
\end{pmatrix}
\]
defined using the real sequences $\{a_n\}_{n=1}^\infty$ and $\{b_n\}_{n=1}^\infty$, if $\vert b_n\vert$ diverges and 
\[
\limsup_{n\to\infty} \frac{a_n^2+a_{n-1}^2}{b_n^2}<\frac{1}{2},
\]
then the operator is self-adjoint with a discrete spectrum. With $a_k=-1$ and $b_k = \alpha k^2 + 2$, $b_k\to \infty$ and the limit supremum of the ratio is $0$. Therefore $B_\alpha^\mc{E}$ and $B_\alpha^{\mc{O}}$ are self-adjoint and each has a discrete spectrum. Positive semi-definiteness follows from the fact that the left-most quantity in the two equalities is always non-negative.
\end{proof}\\

We are now in position to state and prove Theorem \ref{thm2}. The proof relies on two technical lemmas that are proven in the remainder of this section. We let $B_\alpha^{(N)}$ denote an infinite truncated tridiagonal matrix with indices in $\mb{Z}$ such that
\[
B_{j,k} = \left\{\begin{array}{cl}
\alpha k^2 + 2 & j=k\text{ and } -N\leq k \leq N\\
-1 & \vert j-k\vert = 1 \text{ and } -N\leq j \leq N \text{ and } -N\leq k \leq N\\
0 & \text{otherwise}
\end{array}\right.
\]
For all $N\geq 1$, we let $B_\alpha^\mc{E}(N)$ and $B_\alpha^\mc{O}(N)$ denote the $N$ by $N$ leading principal minors of $B_\alpha^\mc{E}$ and $B_\alpha^\mc{O}$, respectively.

\begin{theorem}\label{thm2}
The solutions to the minimum mean spatial-spectral quadratic variation (MMSSQV) are limits of isometries of the form
\[
\Phi_N(x) = \sum_{i=1}^d (u_i^T x) \phi_i(N),
\]
where for each $N\geq d$ the collection $\left\{\widehat{ \phi_i(N)}\right\}_{i=1}^d$ is an orthonormal basis of eigenvectors corresponding to $d$ lowest (but non-zero) eigenvalues of the truncated operator $
B_\alpha^{(N)}$. Moreover, $MSSQV(\Phi_N)$ converges to the optimal value of the minimum spatial-spectral quadratic variation function as $N\to \infty$.
\end{theorem}

\begin{proof}
By Lemma \ref{lemEvenOdd}, the framework of Theorem \ref{thm1} applies for $B_\alpha$. On the other hand, Lemma \ref{lemEvenOdd} also applies to the problem using the block-diagonal operator
\[
\tilde B_\delta = \begin{pmatrix}
B_\delta^\mc{E} & 0\\
0 & B_\delta^\mc{O}
\end{pmatrix}.
\]
Observing that the non-trivial eigenvectors for $B_\alpha^{(N)}$ are identifiable with the eigenvectors of the block-diagonal matrix
\[
\begin{pmatrix}
B_\delta^\mc{E}(N) & 0\\
0 & B_\delta^\mc{O}(N)
\end{pmatrix}
\]
it suffices to show that solving the problem in this truncated system approximates the solution in the infinite-dimensional case. Lemma \ref{lemEigVec} establishes that any of the ``lower'' $d$ eigenvectors of this truncated block-diagonal matrix converges to an eigenvector of the infinite-dimensional system by padding with zeros. Additionally, Lemma \ref{lemDiscSpec} ensures that the quantities
\[
\langle B_\delta^\mc{O}(N) \psi_k(N), \psi_k(N)\rangle \to \lambda_k(B_\delta^{\mc{O}})
\]
as $N\to \infty$ where $\psi_k(N)$ is a unit eigenvector of $B_\delta^{\mc{O}}(N)$ with eigenvalue $\lambda_k(B_\delta^{\mc{O}}(N))$. A similar result holds for $B_\delta^{\mc{E}}(N)$. Therefore, we can also ensure that the value of the mean spatial-spectral quadratic variation is attained for $\Phi_N$ as $N\to\infty$.
\end{proof}\\

The rest of this section is used to establish Lemmas \ref{lemDiscSpec} and \ref{lemEigVec}.

\begin{lemma}\label{lemDiscSpec}
The spectrum of $B_\alpha^\mc{E}$ and the spectrum of $B_\alpha^\mc{O}$ are both purely discrete with multiplicity $1$ for all eigenvalues, and $\sigma(B_\alpha^\mc{E})$ and the spectrum of $\sigma(B_\alpha^\mc{O})$ are disjoint from each other. Moreover, for any fixed $k\geq 1$, $\lambda_k(B_\alpha^\mc{E}(N))\downarrow  \lambda_k(B_\alpha^\mc{E})$ and $\lambda_k(B_\alpha^\mc{O}(N))\downarrow  \lambda_k(B_\alpha^\mc{O})$.
\end{lemma}

\begin{proof}
We first show that $B_\alpha^\mc{O}$ has purely discrete spectrum and that every eigenvalue of $B_\alpha^\mc{O}$ has multiplicity 1. Along the way, we will show that $\lambda_k(B_\alpha^\mc{O}(N))$ (defined for all $N\geq k$) is a monotone decreasing sequence converging to $\lambda_k(B_\alpha^\mc{O})$. The same reasoning applies for $B_\alpha^\mc{E}$, so we omit this part of the argument. However, at the end of the proof we demonstrate that $\sigma(B_\alpha^\mc{E})$ and $\sigma(B_\alpha^\mc{O})$ are necessarily disjoint.

To simplify notation, set $Q_N = B_\alpha^\mc{O}(N)$. Observe that $Q_{N+1}$ is a bordered matrix with leading principal $N$ by $N$ submatrix $Q_N$:
\[
Q_{N+1} = \begin{pmatrix}
Q_N & -e_N\\
-e_N^T & \alpha(N+1)^2+2
\end{pmatrix}
\]
where $e_N\in\R^N$ is the $N$th standard orthonormal basis member of $\R^N$ (that is, $0$ for all entries except the $N$th entry which is $1$). Because of this relationship, Cauchy interlacing yields 
\[
\lambda_1(Q_{N+1}) \leq \lambda_1(Q_N)\leq \lambda_2(Q_{N+1}) \leq \lambda_2(Q_N)\leq \cdots\leq \lambda_N(Q_N)\leq \lambda_{N+1}(Q_{N+1}).
\]
It follows that, for $N\geq k$, $\lambda_k(Q_N)$ is a monotone decreasing sequence. On the other hand, if we define
\[
\tilde Q_N = Q_N - e_Ne_N^T,
\]
we have that 
\[
\tilde Q_{N+1} = \begin{pmatrix}
Q_N & -e_N\\
-e_N^T & \alpha(N+1)^2+1
\end{pmatrix}
\]
and hence 
\[
\lambda_1(\tilde Q_N)\leq\lambda_1(Q_N)\leq \lambda_2(\tilde Q_N)\leq\lambda_2(Q_N)\leq\cdots \leq \lambda_N(\tilde Q_N)\leq\lambda_N(Q_N)
\] 
follows from Cauchy interlacing for rank-one perturbations.  But also observe that if $v$ and $\tilde v$ are unit eigenvectors of $Q_N$ and $\tilde Q_N$ with eigenvalues $\lambda = \lambda_k(Q_N)$ and $\tilde\lambda = \lambda_k(\tilde Q_N)$, then
\[
\lambda v^T\tilde v = \tilde\lambda v^T \tilde v - (e_N^T v)(e_N^T \tilde v)
\]
so
\[
(\lambda - \tilde\lambda) v^T v = - (e_N^T v)(e_N^T \tilde v).
\]
By way of contradiction, suppose $\lambda=\tilde\lambda$. Then either $e_N^Tv=0$ or $e_N^T \tilde v=0$, but then backsolving one of the equations
\[
(Q_N-\lambda I)v=0\text{ or } (\tilde Q_N-\tilde \lambda I)\tilde v
\]
will yield $v=0$ or $\tilde v=0$, contradicting the unit-norm condition. We conclude that $\lambda_k(\tilde Q_N)<\lambda_k(Q_N)<\lambda_{k+1}(\tilde Q_N)$ for $N\geq k+1$.

On the other hand, we have that
\[
\tilde Q_{N+1} = \begin{pmatrix}
\tilde Q_N & 0\\
0 & \alpha(N+1)^2
\end{pmatrix} + (e_N-e_{N+1})(e_N-e_{N+1})^T,
\]
where in this context $e_N, e_{N+1}\in\R^{N+1}$ are the usual orthonormal basis members. This yields the interlacing conditions
\[
\lambda_1(\tilde Q_N) \leq \lambda_1(\tilde Q_{N+1}) \leq \lambda_2(\tilde Q_{N}) \leq\cdots \leq \lambda_N(\tilde Q_{N+1})\leq \alpha(N+1)^2 \leq \lambda_{N+1}(\tilde Q_{N+1})
\]
and hence $\lambda_k(\tilde{Q}_N)$ is a monotone increasing sequence for $k\geq N$.

Collecting our observations, for arbitrary $k\geq 1$,
\[
\lambda_k(Q_N) < \lambda_{k+1}(\tilde Q_N) < \lambda_{k+1}(Q_N)
\]
for all $N\geq k+1$. Monotonicity allows us to set $\lambda_k = \lim_{N\to\infty}\lambda_k(Q_N)$, $\tilde\lambda_{k+1} = \lim_{N\to\infty}\lambda_{k+1}(\tilde Q_N)$ and $\lambda_{k+1} = \lim_{N\to\infty}\lambda_{k+1}(Q_N)$. Since $\lambda_k(Q_N)$ decreases to $\lambda_k$, we get
\[
\lambda_k < \lambda_{k+1}(\tilde Q_N).
\]
Since $\lambda_{k+1}(\tilde Q_N)$ increases to $\tilde \lambda_{k+1}$, we have
\[
\lambda_k < \tilde\lambda_{k+1}.
\]
Finally, $\lambda_{k+1}(\tilde Q_N) < \lambda_{k+1}(Q_N)$ for all $N\geq k+1$ implies $\tilde\lambda_{k+1}\leq \lambda_{k+1}$, and the transitive property yields $\lambda_k<\lambda_{k+1}$ for all $k\geq 1$. On the other hand, the interlacing condition $\alpha k^2 \leq \lambda_k(\tilde Q_k)$ (seen above as $\alpha(N+1)^2 \leq \lambda_{N+1}(\tilde{Q}_{N+1})$) implies $\alpha k^2 \leq \tilde\lambda_k\leq \lambda_k$ for $k\geq 2$. 

Now, set
\[
\Lambda = \left\{\lambda: \lambda=\lim_{N\to\infty} \lambda^{(N)}\text{ where } \lambda^{(N)}\in\sigma(Q_N)\right\}.
\]
That is, $\Lambda$ is the set of limits of eigenvalues of $Q_N$ as $N\to\infty$. We will show that $\Lambda= \{\lambda_1, \lambda_2, \ldots\}$ where $\lambda_k$ are defined above. Let $\lambda^{(N)} = \lambda_{k_N}(Q_N)$ and suppose $\lambda = \lim_{N\to\infty}\lambda^{(N)}$. First, it must be the case that $k_N$ is bounded. By way of contradiction, suppose there is a subsequence $k_{N_j}$ such that $k_{N_j}\to\infty$ as $j\to\infty$. Then 
\[
\alpha (k_{N_j})^2 \leq \lambda_{k_{N_j}}(Q_{N_j})
\]
yields that the subsequence $\lambda_{k_{N_j}}(Q_{N_j})$ diverges as $j\to\infty$, a contradiction to the fact that $ \lambda_{k_N}(Q_N)$ converges to $\lambda$. 

Next, we claim that $k_N$ is eventually constant. By the pigeonhole principle, since $k_N$ is bounded, there is a $k_1$ such that $k_N=k_1$ for infinitely many $N$. By way of contradiction, if there is a $k_2\not=k_1$ such that $K_N=k_2$ for infinitely many $N$, then there are subsequences $k_{N_{1,j}}$ and $k_{N_{2,j}}$ which are eventually $k_1 \not= k_2$ (respectively), then
\[
\lambda_{k_{N_{1,j}}}(Q_{N_{1,j}})\to \lambda_{k_1}
\]
and 
 \[
\lambda_{k_{N_{2,j}}}(Q_{N_{2,j}})\to \lambda_{k_2}.
\]
But by the above, $k_1<k_2$ then $\lambda_{k_1}<\lambda_{k_2}$, and if $k_2<k_1$ $\lambda_{k_1}<\lambda_{k_2}$. Either way, we contradict he fact that subsequences of a convergent sequence are convergent to the same limit. We therefore conclude $k_1$ is the only $k$ between $1$ and $\sup k_N<\infty$ for which $k_N=k_1$ infinitely many times. Thus, there is a finite index $M$ such that $k_N=k_1$ for all $N\geq M$. Consequently, $\lambda_{k_1}=\lambda$. 

Theorem 2.4 of \cite{petropoulou2014self} indicates that $\Lambda = \sigma(B_\alpha^{\mc{O}})$ since the diagonal entries of $B_\alpha^{\mc{O}}$ diverge and $B_\alpha^{\mc{O}}$ also satisfies condition (2.2) of that paper. We conclude that $B_\alpha^{\mc{O}}$ has a discrete spectrum. To prove that $\lambda_k=\lambda_k(B_\alpha^{\mc{O}})$, suppose $B_\alpha^{\mc{O}} v = \lambda_k v$. If $v[1]=0$, then we obtain a contradiction because we inductively conclude $v[k]=0$ for all $k\in\mb{N}$. If $v[1]\not=0$, replace $v$ with $v[1]^{-1} v$ so that $v[1]=1$. Then solving the system inductively necessarily yields unique values. Therefore the dimension of the kernel of $B_\alpha^{\mc{O}}$ is $1$. 

Finally, we observe that 
\[
B_\alpha^{\mc{E}} = \begin{pmatrix} 
2 & -\sqrt{2} e_2^T\\
-\sqrt{2}e_2^T & B_\alpha^{\mc{O}}
\end{pmatrix}
\]
So given two unit eigenvectors, we have
\[
\lambda v^T u = \mu v^T u -\sqrt{2}(e_1^T v)(e_2^T u).
\]
Then $\lambda=\mu$ would imply either $e_1^T v=0$ or $e_2^Tu=0$ which would lead to either $v=0$ or $u=0$ by induction. This contradicts the the fact that $v$ and $u$ were chosen to be unit vectors. 
\end{proof}\\

The next lemma ensures that we can approximate the eigenvectors of $B_\alpha^\mc{E}$ and $B_\alpha^\mc{O}$ using the eigenvectors of $B_\alpha^\mc{E}(N)$ and $B_\alpha^\mc{O}(N)$, and therefore we obtain an optimal isometry. 

\begin{lemma}\label{lemEigVec}
For any fixed $k\geq 1$, given any sequence of unit vectors $\{v_N\}_{N=1}^\infty$ such that $v_N\in\R^N$ is an eigenvector of $B_\alpha^{\mc{O}}(N)$, there is a choice of signs $\xi_N\in\{-1,1\}$ such that the padded unit vectors
\[
\tilde v_N = \xi_N\begin{pmatrix}
v_N\\
0\\
0\\
\vdots
\end{pmatrix}\in\ell^2(\mb{N})
\]
converge to a unit vector $v\in\ell^2(\mb{N})$ which is an eigenvector for $B_\alpha^{\mc{O}}$ with eigenvalue $\lambda_k(B_\alpha^{\mc{O}})$. A similar result holds for $B_\alpha^{\mc{E}}$. 
\end{lemma}

\begin{proof}
Set $Q_N=B_\alpha^{\mc{O}}(N)$ to simplify notation, and let $\{v_j(N)\}_{j=1}^N$ denote an orthonormal basis of eigenvectors of $Q_N$ associated with the non-decreasing eigenvalues $\{\lambda_j(Q_N)\}_{j=1}^N$ in a respective fashion. By Lemma \ref{lemDiscSpec}, for any $k\geq 1$, $\lambda_{k}(Q_N)\downarrow \lambda_{k}(B_\alpha^{\mc{O}})=\lambda_k$ and $\lambda_{k+1}(Q_N)\downarrow \lambda_{k+1}(B_\alpha^{\mc{O}})=\lambda_{k+1}$, and $\lambda_k\not=\lambda_{k+1}$. Therefore, there is a $\delta>0$ and an $M\in\mb{N}$ such that
\[
\delta < \lambda_{k+1} - \lambda_{k}(Q_N) < \lambda_{k+1}(Q_N) - \lambda_{k}(Q_N)
\]
for all $N\geq M$. By taking a minimum over $\delta$'s and a maximum over $M$'s, we can obtain a $\delta>0$ and an $M\in\mb{N}$ such that
\[
\delta < \lambda_{k+1} - \lambda_{k}(Q_N) < \lambda_{k+1}(Q_N) - \lambda_{k}(Q_N)
\]
and
\[
\delta <  \vert \lambda_{k}(Q_N) - \lambda_j\vert  < \vert \lambda_{k}(Q_N) - \lambda_{j}(Q_N)\vert
\]
for all $j<k$. 

We note that 
\[
(\alpha N^2+2-\lambda_k(Q_N))(e_{N}^Tv_k(N))-(e_{N-1}^Tv_k(N)) =0
\]
so 
\[
e_{N}^Tv_k(N) = \frac{e_{N-1}^Tv_k(N)}{\alpha N^2+2-\lambda_k(Q_N)}
\]
and
\[
-(e_{N-2}^Tv_k(N))+(\alpha (N-1)^2+2-\lambda_k(Q_N))(e_{N-1}^Tv_k(N))-(e_{N}^Tv_k(N)) =0.
\]
Then 
\[
e_{N-1}^Tv_k(N) = \frac{(e_{N-2}^Tv_k(N))+(e_{N}^Tv_k(N))}{\alpha (N-1)^2+2-\lambda_k(Q_N)}
\]
and combining these we have
\[
\vert e_{N}^Tv_k(N)\vert \leq \frac{2}{(\alpha N^2+2-\lambda_k(Q_N))(\alpha (N-1)^2+2-\lambda_k(Q_N))}
\]
Let $M$ be such than $N\geq M$ implies $\lambda_k(Q_N)-\lambda_k < 2$ and $\alpha(N-1)^2> \lambda_k$. We have the bound
\[
\vert e_{N}^Tv_k(N)\vert \leq \frac{2}{(\alpha N^2-\lambda_k)(\alpha (N-1)^2-\lambda_k)} \leq \frac{2}{(\alpha (N-1)^2-\lambda_k)^2} 
\]

Now,
\begin{align*}
\lambda_j(Q_{N+1}) v_j(N+1)^T \tilde v_k &= v_j(N+1) Q_{N+1} \tilde v_k\\
&= \lambda_k(Q_{N}) v_j(N+1)^T \tilde v_k - (e_N^T \tilde v_k)(e_{N+1}^T v_j(N+1)),
\end{align*}
so we have
\[
v_j(N+1)^T \tilde v_k = -\frac{(e_N^T \tilde v_k)(e_{N+1}^T v_j(N+1))}{\lambda_j(Q_{N+1})-\lambda_k(Q_{N})}
\]
and therefore
\[
\vert v_j(N+1)^T \tilde v_k\vert^2 \leq \frac{(e_N^T \tilde v_k)^2}{\delta^2}.
\]
Hence,
\[
\vert v_k(N+1)^T \tilde v_k(N)\vert^2 = 1 - \sum_{j\not=k} \vert v_j(N+1)^T \tilde v_k\vert^2 \geq 1 - \vert e_N^T \tilde v_k\vert^2\sum_{j\not=k} \frac{1}{\delta^2} = 1 - \frac{N}{\delta^2}\vert e_N^T \tilde v_k\vert^2.
\]
Therefore, letting $\xi_N=\text{sign}(v_k(N+1)^Tv_k(N))$,
\begin{align*}
\Vert \xi_Nv_k(N+1) - \tilde v_k(N)\Vert^2 &= 2 - 2\vert v_k(N+1)^Tv_k(N)\vert\\
& < 2\left (1-\sqrt{1 - \frac{N}{\delta^2}\vert e_N^T \tilde v_k\vert^2}\right)\\
&= \frac{2N\vert e_N^T \tilde v_k\vert^2}{\delta^2\left (1+\sqrt{1 - \frac{N}{\delta^2}\vert e_N^T \tilde v_k\vert^2}\right)}
\end{align*}
and hence
\[
\Vert \xi_Nv_k(N+1) - \tilde v_k(N)\Vert\leq \frac{\sqrt{2N}\vert e_N^T \tilde v_k(N)\vert}{\delta}. \leq \frac{2\sqrt{2}}{\delta} \frac{\sqrt{N}}{(\alpha(N-1)^2-\lambda_k)^2}.
\]
Observe that the terms on the left are $\mc{O}(N^{-7/2})$, so by the comparison test and convergence of $p$-series, we have that
\[
\sum_{N=M}^\infty \Vert \xi_{N+1}\tilde v_k(N+1) - \xi_N\tilde v_k(N)\Vert < \frac{2\sqrt{2}}{\delta} \sum_{N=M}^\infty \frac{\sqrt{N}}{(\alpha(N-1)^2-\lambda_k)^2}<\infty
\]
Therefore
\[
\xi_{K+1}\tilde v_k(K+1) - \xi_M \tilde v_k(M)=\sum_{N=M}^K \xi_{N+1}\tilde v_k(N+1) - \xi_N\tilde v_k(N)\to u
\]
as $K\to\infty$, so we have that $\xi_{K+1}\tilde v_k(K+1)\to v$ in $\ell^2(\mb{N})$ as $K\to\infty$. By continuity of the norm, we have $\Vert v \Vert=1$. 

Now we show that $v$ is an eigenvector of $B_\alpha^{\mc{O}}$ with eigenvalue $\lambda_k$. We have that
\[
B_\alpha^{\mc{O}} \tilde v_k(N) = \lambda_k(N)\tilde v_k(N) - (e_N^T v_k(N))e_{N+1}.
\]
so
\begin{align*}
\Vert B_\alpha^{\mc{O}} \tilde v_k(N) - \lambda_k v\Vert &\leq \Vert \lambda_k(N)\tilde v_k(N) - \lambda_k v\Vert + \vert e_N^T v_k(N)\vert\\
& \leq \lambda_k(N)\Vert v_k(N) - v\Vert + \vert \lambda_k(N)-\lambda_k\vert \Vert v\Vert + \vert e_N^T v_k(N)\vert
\end{align*}
Since each term on the right hand side converge to $0$ as $N\to\infty$, we conclude that $\lim_{N\to\infty} B_\alpha^{\mc{O}} \tilde v_k(N)= \lambda_k v$. 
By Lemma \ref{lemEvenOdd}, $B_\alpha^{\mc{O}}$ is self-adjoint and hence closed. We have verified that $\tilde v_k(N)\to v$ and $B_\alpha^{\mc{O}} \tilde v_k(N)\to \lambda_k v$, so $B_\alpha^{\mc{O}} v= \lambda_k v$ because $B_\alpha^{\mc{O}}$ is a closed operator.

\end{proof}

\section{Examples}

In this section we discuss further examples including the Wisconsin breast cancer dataset \cite{street1993nuclear} and the diabetes dataset \cite{efron2004least}. Often, 2D scatterplots reveal clusters within datasets. To imitate this in $\Lt$, observe that any linear isometry $\Phi:\R^d\to \Lt$ preserves convex hulls, and for any set of functions $\{f_i\}_{i=1}^n\subset\Lt$ the convex hull of this collection in $\Lt$ is contained in the set
\[
\{f\in \Lt: \min\{f_i(t)\}_{i=1}^n\leq f(t)\leq \max\{f_i(t)\}_{i=1}^n\text{ for all }t\in[0,1]\}.
\]
Therefore, to visualize separation of classes, we visualize ``bands'' demarcated by the ``upper'' and ``lower'' envelope functions
\[
u(t) = \max\{f_i(t)\}_{i=1}^n \text{ and }l(t) =  \min\{f_i(t)\}_{i=1}^n.
\]
We illustrate these bands for the iris dataset in Figures \ref{fig3} and  \ref{fig4}

\begin{center}
\begin{figure}
\centering
\includegraphics[scale=0.4]{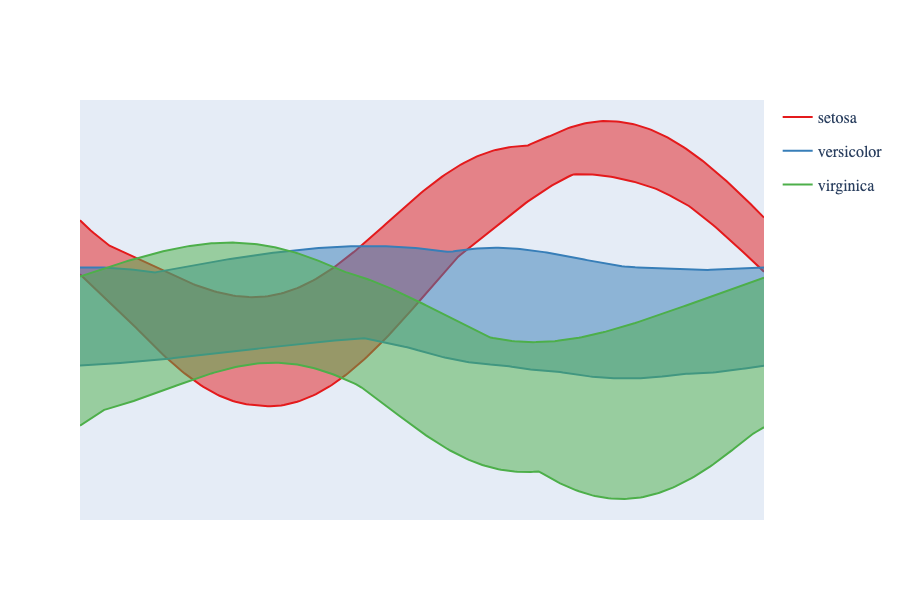}
\caption{Bands for the Andrews plots of the Iris dataset.}
\label{fig3}
\end{figure}
\end{center}

\begin{center}
\begin{figure}
\centering
\includegraphics[scale=0.4]{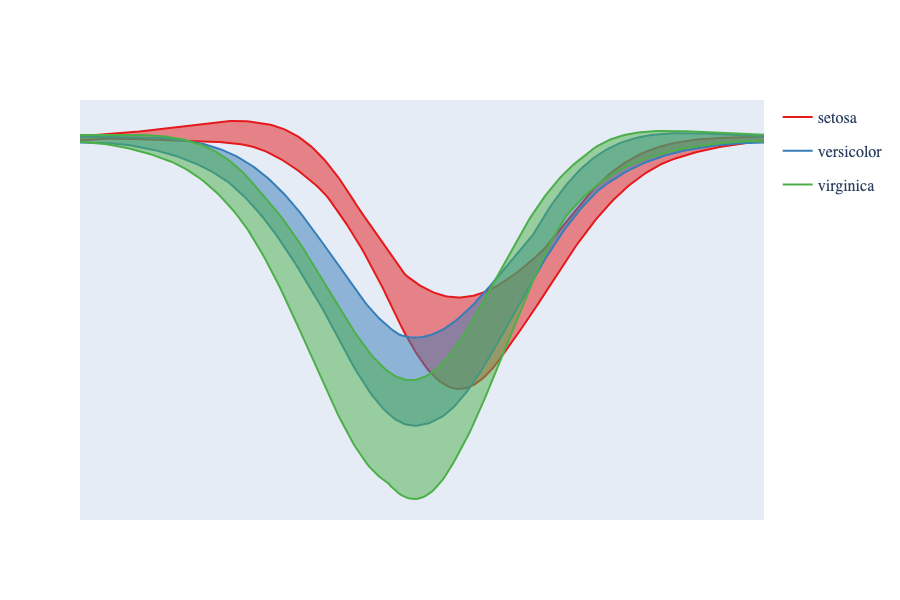}
\caption{Bands of Andrews plots with spatial-spectral smoothing for the Iris dataset.}
\label{fig4}
\end{figure}
\end{center}

\subsection{Breast cancer dataset}
The Wisconsin breast cancer dataset consists of 569 examples in 30 dimensions. There are two classes: malignant and benign tumors. Figures \ref{fig5}, \ref{fig6}, \ref{fig7}, and \ref{fig8} depict the different Andrews plots and the ``bands'' for the two different classes. Using these plots, we verify that the two classes admit some sort of overlap, but the bulk of the two classes seem well separated from each other. We note that the Andrews plots with spatial-spectral smoothing achieve a more-localized largest gap between the different classes. This ultimately leads to a more interpretable plot since the largest gap (i.e. $L^\infty([0,1])$ distance is easy to visualize.

\begin{center}
\begin{figure}
\centering
\includegraphics[scale=0.4]{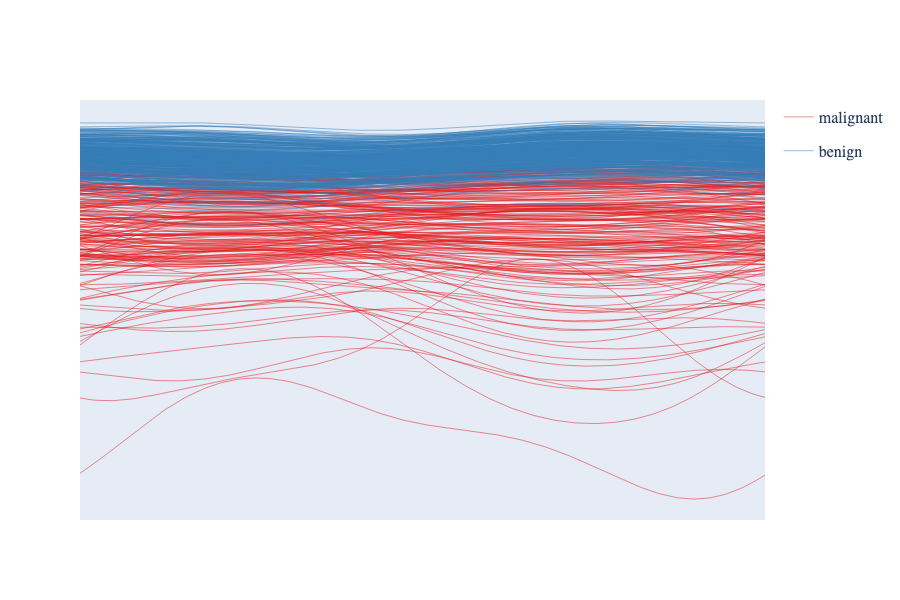}
\caption{Andrews plots of the Wisconsin breast cancer dataset.}
\label{fig5}
\end{figure}
\end{center}

\begin{center}
\begin{figure}
\centering
\includegraphics[scale=0.4]{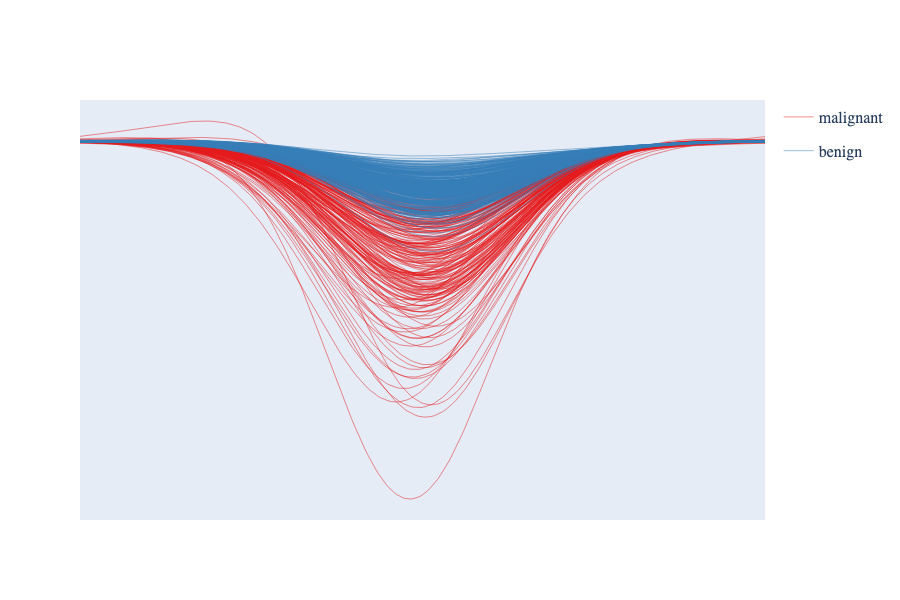}
\caption{ Andrews plots with spatial-spectral smoothing for the Wisconsin breast cancer dataset.}
\label{fig6}
\end{figure}
\end{center}

\begin{center}
\begin{figure}
\centering
\includegraphics[scale=0.4]{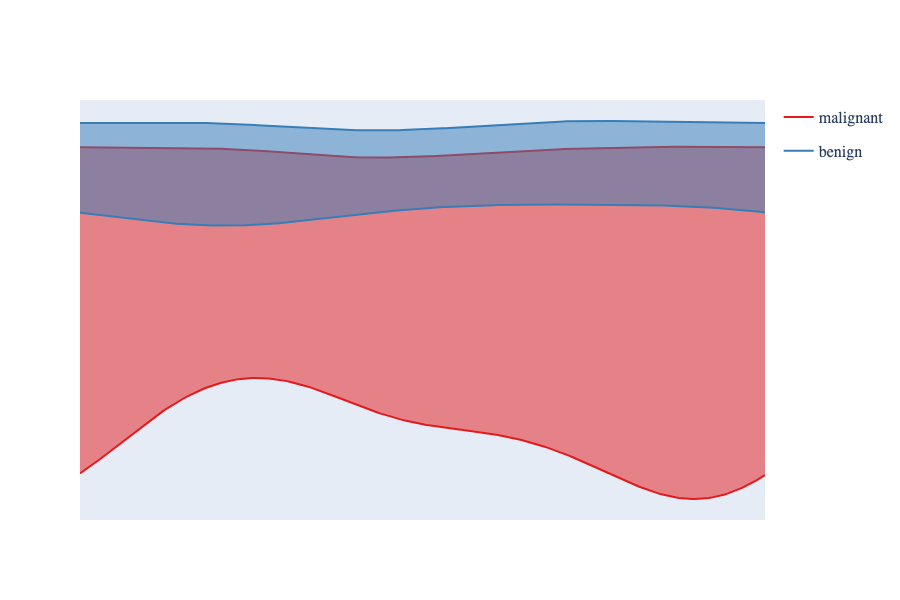}
\caption{Bands for the Andrews plots of the Wisconsin breast cancer dataset.}
\label{fig7}
\end{figure}
\end{center}

\begin{center}
\begin{figure}
\centering
\includegraphics[scale=0.4]{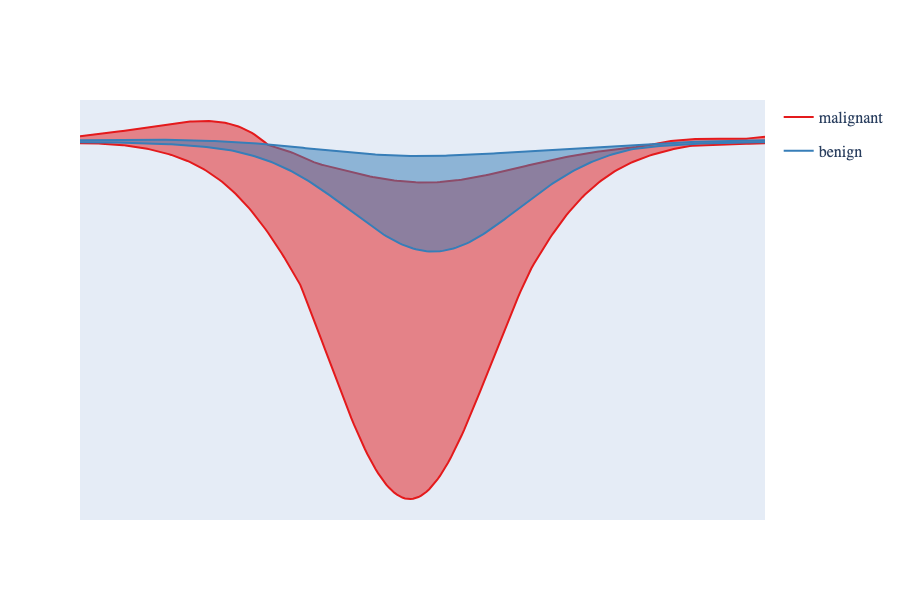}
\caption{Bands of Andrews plots with spatial-spectral smoothing for the Wisconsin breast cancer dataset.}
\label{fig8}
\end{figure}
\end{center}

\subsection{Diabetes dataset}
The diabetes dataset consists of 442 examples in 10 dimensions. The target variables are integers in the range of $25$ to $356$. We replace these targets with indicators for quartile ranges. That is, we set the target variable to $Q1$ if the original target is in the first quartile and so on. Figures \ref{fig9}, \ref{fig10}, \ref{fig11}, and \ref{fig12} depict the different Andrews plots and the ``bands" for the data coming from the different ranges.  We note that the Andrews plots with spatial-spectral smoothing exaggerate and localize differences as compared to the standard Andrews plots. For regression problems, we do not fully expect separable clusters to emerge, but the envelopes for the different quartiles are exhibit distinct features.

\begin{center}
\begin{figure}
\centering
\includegraphics[scale=0.4]{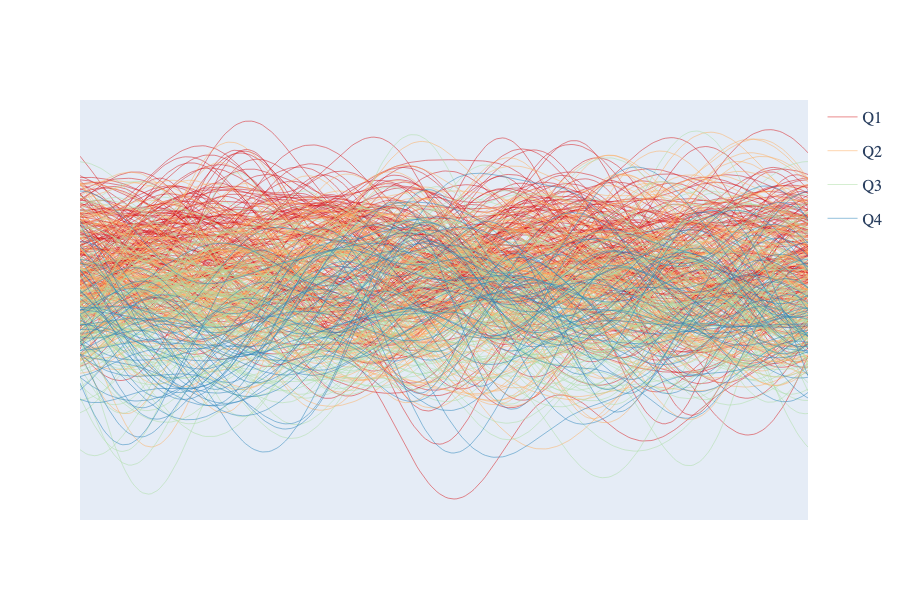}
\caption{Andrews plots of the diabetes dataset.}
\label{fig9}
\end{figure}
\end{center}

\begin{center}
\begin{figure}
\centering
\includegraphics[scale=0.4]{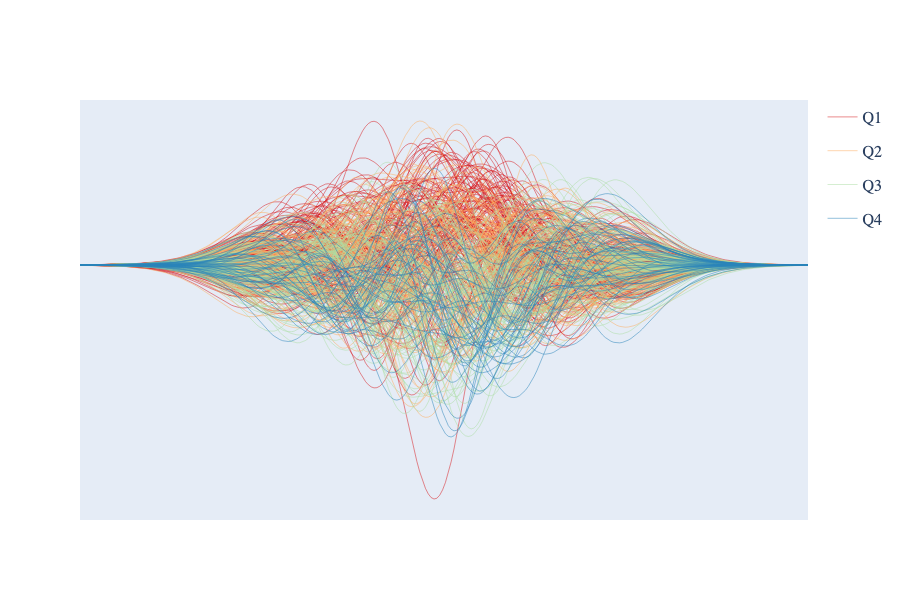}
\caption{Andrews plots with spatial-spectral smoothing for the diabetes dataset.}
\label{fig10}
\end{figure}
\end{center}

\begin{center}
\begin{figure}
\centering
\includegraphics[scale=0.4]{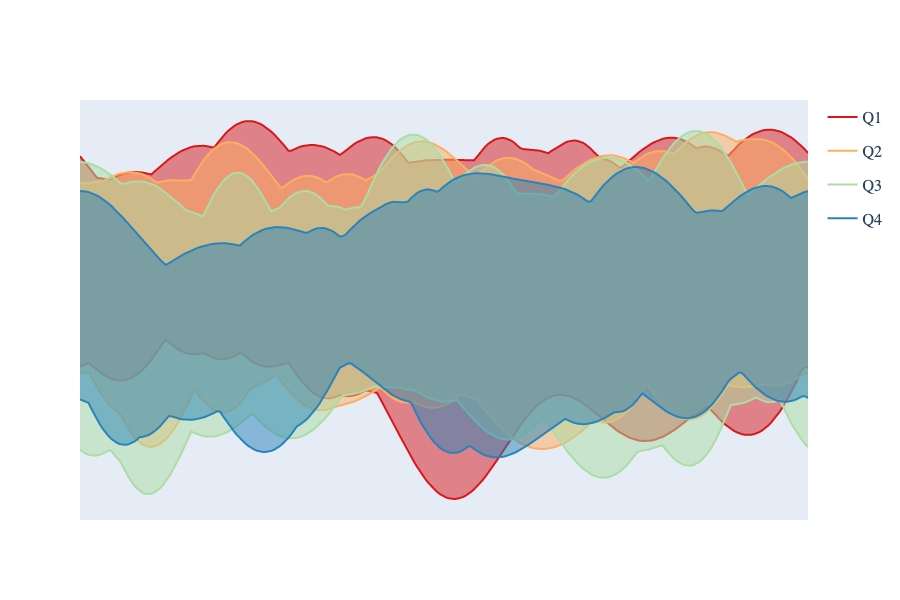}
\caption{Bands for the Andrews plots of the diabetes dataset.}
\label{fig11}
\end{figure}
\end{center}

\begin{center}
\begin{figure}
\centering
\includegraphics[scale=0.4]{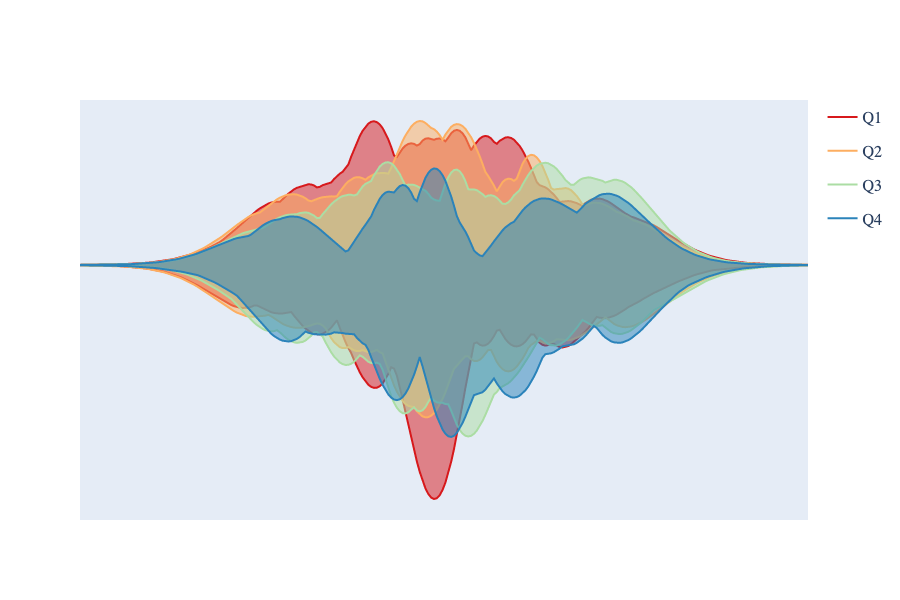}
\caption{Bands of Andrews plots with spatial-spectral smoothing for the diabetes dataset.}
\label{fig12}
\end{figure}
\end{center}

\section{Conclusion}

Since we employ PCA coefficients, it is natural to consider how these maps perform with kernel PCA. \cite{scholkopf1998nonlinear}. For future work, we are interested in determining how spatial-spectral smoothing assists visualization using kernel PCA coefficients. 

In our examples, we observe that the $L^\infty([0,1])$ distance is pronounced for the Andrews plots with spatial-spectral smoothing. In future work, we would like to determine if functions with finite expansions in the $B_\delta$ basis admit better $L^\infty([0,1])$ than trigonometric polynomials. For a trigonometric polynomial $f$, we know that
\[
\frac{C}{\sqrt{d}}\Vert f \Vert_{L^\infty([0,1])} \leq \Vert f\Vert_{\Lt} \leq \Vert f \Vert_{L^\infty([0,1])}. 
\]
where $d$ is the maximum degree of the trigonometric polynomial $f$. This follow from \cite{erdos1962inequality}.

Finally, there is one remaining ambiguity in even the Andrews plots with spatial-spectral smoothing due to the sign of the eigenfunctions. This means that there are $2^{d-1}$ possible optimal plots. In the future, it would be interesting to find a meaningful choice of signs for the purposes of visualization,

\section*{Acknowledgments}
We would like to thank Jameson Cahill and Mark Lammers for helpful suggestions and discussions.

\bibliographystyle{plain}
\bibliography{numericalAndrews}
\end{document}